\documentclass[12pt]{article}
\usepackage[top=1in, bottom=1in, left=1in, right=1.25in]{geometry}
\usepackage{amsmath}
\usepackage{enumerate}
\usepackage{fancyvrb}
\usepackage{amssymb}
\usepackage{color}
\usepackage{mathrsfs}

\usepackage{mathabx}
\usepackage{amsthm}
\usepackage[dvipsnames]{xcolor}
\usepackage{wasysym}
\usepackage{titlesec}
\titleformat{\subsection}
        {\normalfont\bfseries}
        {\thesubsection}
        {0.5em}
        {}
        []

\usepackage{tikz}
\usetikzlibrary{cd}
\usetikzlibrary{decorations.markings}

\theoremstyle{plain}
\newtheorem{prop}{Proposition}
\newtheorem{thm}{Theorem}
\newtheorem{lemma}{Lemma}

\newtheorem*{rmk}{Remark}

\def\C{\mathbb{C}}

\def\Z{\mathbb{Z}}

\def\I{\mathcal{I}}

\def\K{\mathcal{K}}
\def\O{\mathcal{O}}

\def\span{\operatorname{Span}}

\def\ord{\operatorname{ord}}

\def\Im{\operatorname{Im}}

\newcommand{\tup}[1]{\langle{#1}\rangle}

\title{A Kazhdan-Lusztig Atlas on $G/P$}
\author{Daoji Huang}
\date{}
\begin{document}
\maketitle



\begin{abstract}
A stratified variety has a Kazhdan-Lusztig atlas if it can
be locally modelled with Kazhdan-Lusztig varieties 
stratified by Schubert varieties in some
Kac-Moody flag manifold via stratified isomorphisms.
In this paper, we show that the partial flag manifold
$G/P$ with the projected Richardson stratification has
a Kazhdan-Lusztig atlas, with each chart stratified-isomorphic to a Kazhdan-Lusztig variety in the
affine flag manifold of the formal loop group $\widehat{G}$ of $G$. This result generalizes
that of Snider's on $Gr(k,n)$ with the positroid stratification, and is a geometric counterpart of the
combinatorial correspondence between the poset of
projected Richardson stratification and a certain convex set
in the Bruhat order of the Weyl group of $\widehat{G}$ given by He and Lam.

\end{abstract}
\section{Introduction and Statements of the Main Result}

\subsection{Projected Richardson stratification on $G/P$}
Let $G$ be a complex, connected, and simply connected
reductive group over $\mathbb{C}$.
Let $B$ and $B_-$ be the upper and lower
Borel subgroups 
containing the maximal torus $T$,
and $U$ (resp. $U_-$) denote the unipotent radical of 
$B$ (resp. $B_-$). 
Let $X_*(T)$ denote the coweight lattice of $T$. 
Let $W=N(T)/T$ denote the Weyl group of $G$, and
$w_0\in W$ its longest element. 
Let $S$ denote the set of simple roots determined
by $(T,B)$. 
Let $\Phi$ denote the set of all roots of $G$,  and $\Phi_+$ (resp. $\Phi_-$)  the set of positive (resp. negative)
roots. 
For any $\alpha\in \Phi$, let $s_\alpha$ denote the corresponding 
reflection. When $\alpha\in S$,
$s_\alpha$ is a simple reflection.

The flag variety $G/B$ has a
stratification by Schubert
cells $G/B=\bigsqcup_{w\in W} X_w^\circ$
where $X_w^\circ = B_-wB/B$ and
 by opposite Schubert cells $G/B=\bigsqcup_ {w\in W} X^w_\circ$ where $X^w_\circ = BwB/B$. Let $X_w=\overline{B_-wB/B}$ denote the Schubert variety and 
 $X^w=BwB/B$ denote the opposite Schubert variety. The intersections of the Schubert and
 opposite Schubert cells are called  open Richardson varieties, denoted by
 $\mathring{X}^w_v=X^w_\circ\cap X_v^\circ$. These also form a stratification
 of $G/B$, and this stratification is anticanonical.
 The closure of $\mathring{X}^w_v$ is the Richardson variety $X^w_v=X^w\cap X_v$.

We fix a subset $J$ of the simple roots $S$. Let $P$ denote the
parabolic subgroup corresponding to $J$, and
$P_-$ the opposite parabolic subgroup of $P$. Let $W_P$
denote the subgroup of $W$ generated by
$\{s_\alpha:\alpha\in J\}$, and $W^P$ the set of minimal coset 
representatives of $W/W_P$. 
Let $w_{0,P}$ denote the longest element in $W_P$. 
Let $U^P$ (resp. $U^P_-$)
denote the unipotent radical of $P$ (resp. $P_-$). 
$P$ admits a Levi decomposition $P=L\ltimes U^P$
where $L\supset T$. 

Let $\pi$ denote the projection $G/B\to G/P$. 
Denote by 
$\mathring{\Pi}^w_v$ the open projected Richardson
variety $\pi(X^w_\circ\cap X_v^\circ)$ and $\Pi^w_v$ the projected Richardson variety $\pi(X^w\cap X_v)$.
$G/P$ admits a projected Richardson stratification
\[G/P=\bigsqcup_{(w,v)\in Q_P}\mathring{\Pi}_v^w,\]
where $Q_P=\{(w,v):w\in W^P, v\in W, v\le w\}.$
This is an anticanonical stratification of $G/P$.
Define an ordering $\preceq$ on $W^P\times W$
$(w',v')\preceq (w,v)$ if and only if there exists $u\in W_P$ such that 
$w'u\le w$ and $v'u\ge v$. Then $(Q_P,\preceq)$ is a subposet of $(W^P\times W, \preceq)$, and for any $(w',v'), (w,v)\in Q_P$, $(w',v')\preceq (w,v)$
if and only if $\Pi^{w'}_{v'}\subseteq \Pi^w_v$.  For more details
on projected Richardsons, see
\cite{KLS}.

\subsection{Formal loop group, affine root system, and affine flag variety}
Let $\mathcal{K}$ denote the local field of Laurent
series $\C((t^{-1}))$, and 
$\mathcal{O}=\C[[t^{-1}]]$ its ring of formal power series. Also let
$\mathcal{O}_-$ be the polynomial ring $\C[t]$.
 The formal loop group 
 $G(\mathcal{K})$ 
 (also written as $\widehat{G}$) has an extended torus
$T\times \mathbb{C}^*$ where $\mathbb{C}^*$ acts via loop rotation by scaling
$t$. 
A character on $T$ can be identified with a character
on $T\times \mathbb{C}^*$ that is trivial on $\mathbb{C}^*$. Furthermore, let
$\delta$ denote the character on $T\times \mathbb{C}^*$ that is trivial on
$T$ and identity on $\mathbb{C}^*$. 
The affine root system of $G(\mathcal{K})$ is given by
\[\Phi_\text{aff}=\{\alpha + m\delta: \alpha\in \Phi\sqcup\{0\}, m\in \mathbb{Z}\}\setminus\{0\}.\] 
and the simple roots
$S_\text{aff}=S\sqcup\{\alpha_0-\delta\}$
where $\alpha_0\in\Phi$ is the lowest root. 
$G(\mathcal{O})$ is a maximal parabolic subgroup of $G(\mathcal{K})$. Its unipotent subgroup $\mathcal{U}^{G(\mathcal{O})}$ corresponds to the
set of roots $\alpha+m\delta$ where $m<0$. Similarly, 
$\mathcal{U}^{G(\mathcal{O}_-)}$ corresponds to the set of roots
$\alpha+m\delta$ where $m>0$. 

The affine Weyl group 
$\widehat{W}=W\ltimes X_*(T)$ acts on $\Phi_\text{aff}$, where $W$ acts on
$\delta$ trivially and $\alpha\in \Phi$ the usual way. Also for any $t^\mu\in X_*(T)$, \[t^\mu\cdot (\alpha +m\delta)=\alpha + (m+\langle \mu, \alpha\rangle )\delta.\] 
Let 
$\mathcal{I}=\{f\in G(\mathcal{O}):f(\infty)\in B\subset G\}$ be the standard
Iwahori subgroup 
and
$\mathcal{I}_-=\{f\in G(\mathcal{O}_-):f(0)\in B_-\subset G\}$
the opposite Iwahori subgroup of $\mathcal{I}$.
The affine flag variety
is $G(\mathcal{K})/\mathcal{I}$.
The Schubert and opposite
Schubert varieties
are defined as
$\mathcal{X}_w = \overline{\mathcal{I}_-w\mathcal{I}/\mathcal{I}}$
and $\mathcal{X}^w=\overline{\mathcal{I}w\mathcal{I}/\mathcal{I}}$
where $w\in 
\widehat{W}$.

\subsection{Kazhdan-Lusztig atlas} 
Let $(H,  B^\pm_H, T_H, W_H)$ be a pinning of a 
Kac-Moody Lie group. A \textbf{Kazhdan-Lusztig variety} in the Kac-Moody
flag variety $H/B_H$ is the intersection of
an opposite Schubert cell
$X^w_\circ =B_H wB_H/B_H$ and a Schubert variety $X_v=\overline{B_HvB_H/B_H}$, where
$w,v\in W_H$.  
Let $M$ be a smooth variety bearing a stratification $\mathcal{Y}$, 
whose minimal
strata $\mathcal{Y}_{\min}$ are points. To be precise,
$M=\bigsqcup_{y\in \mathcal{Y}}M_y^\circ$ where
$\mathcal{Y}$ is a ranked poset,
each
$M_y^\circ$ is a locally closed subvariety
with closure $M_y$,
and  $M_y=\bigsqcup_{y'\le y} M_{y'}^\circ$.
We say that $(M,\mathcal{Y})$ has
a \textbf{Kazhdan-Lusztig atlas} 
with the modelling Kac-Moody
$H/B_H$ if
\begin{enumerate}[(a)]
\item there is a ranked poset injection
$v:\mathcal{Y}^{\text{opp}}\to W_H$ whose image is
$\bigcup_{f\in\mathcal{Y}_{\min}} [v(M), v(f)]$, 
\item $M$ can be covered by charts $U_f$ centered around the minimal strata $f\in \mathcal{Y}_{\min}$ such that for each $f$,
there is a stratification-preserving isomorphism
\[c_f:U_f\to X^{v(f)}_\circ\cap X_{v(M)}\subset H/B_H.\]
\end{enumerate}

\subsection{Main result} Our goal is to show that 
 $G/P$ equipped
with the projected Richardson stratification has a Kazhdan-Lusztig
atlas with the modelling flag
variety $G(\mathcal{K})/\mathcal{I}$.
The combinatorial part of the result concerning the poset structures has been
established by He and Lam \cite{HL}. Our result
is a generalization of Snider's
result on the Grassmannian 
variety $Gr(k,n)$ of $k$-planes in $n$-space
with the positroid stratification \cite{Snider}. 
In that case, $Gr(k,n)$ admits a \textbf{Bruhat atlas} \cite{HKL}, where each chart is modelled by
an opposite Schubert cell 
(i.e., $v(M)=1$) of the 
corresponding type A
affine
flag variety.
\begin{thm}
$G/P$ with the projected Richardson stratification has a Kazhdan-Lusztig atlas
\[w_1U_-^P P/P\cong \mathcal{X}^{w_1t^{\lambda}w_{0,P}w_0w_2}_\circ\cap \mathcal{X}_{t^\lambda w_{0,P}w_0}\] 
where $w_1\in W^P$, $w_2w_1=w_0w_{0,P}$, and
$\lambda$ is a dominant coweight such that
 for any $\alpha\in S$,
$\langle \alpha, \lambda\rangle = 0$
iff $\alpha\in J$. 

Under the isomorphism, the strata correspond
as follows
\[w_1U_-^P P/P \cap \Pi^w_v \cong \mathcal{X}_\circ^{w_1t^\lambda w_{0,P}w_0w_2}\cap \mathcal{X}_{vt^\lambda w^{-1}}\ (\text{for all }v\in W, w\in W^{P},v\le w)\]
\end{thm}
\begin{rmk}
\normalfont This result has also been established
independently
recently by Galashin, Lam, and Karp \cite{GKL} through
a slightly different map.
\end{rmk}
\section{Chart Isomorphisms and Proofs}
To describe the chart isomorphisms, we
use Bott-Samelson maps.
For a detailed description of
Bott-Samelson variety in full
Kac-Moody generality, see
\cite{kumar2012kac}. The cases relevant to us
are finite or affine type.
For any $w\in W$ (resp. $\widehat{W}$)
and
$Q=(s_{\alpha_1}, \cdots, s_{\alpha_{l(w)}})$ a reduced word for $w$, we
have a Bott-Samelson map $m_Q:\mathbb{A}^{l(w)}\to G$ (resp. $\widehat{G}$) such that $m_Q(z_1,\cdots,z_{l(w)})= \prod_{i=1}^{l(w)}   \exp(z_iR_{\alpha_i})\widetilde{s_{\alpha_i}}$ where $R_{\alpha_i}$ is a 
root vector that spans the root space indexed by
$\alpha_i$. First we establish some lemmas.
(The dot notation ``$\cdot$'' in this paper denotes the
conjugation action on groups and the induced action Lie algebras. 
It takes higher precedence than
binary set operations.)

\begin{lemma}
For any $w\in W^P$, $w\cdot \Phi_{J}\subset \Phi_+$.
\end{lemma}
\begin{proof}
We have $l(w)=|w\cdot(\Phi_-\setminus \Phi_{-J}\sqcup \Phi_{-J})\cap \Phi_+|=|w\cdot(\Phi_-\setminus \Phi_{-J})\cap \Phi_+)|+|w\cdot \Phi_{-J}\cap \Phi_+|$
and
$w_{0,P}\cdot \Phi_-=(\Phi_-\setminus \Phi_{-J})\sqcup \Phi_J$.
Then $l(ww_{0,P})=|ww_{0,P}\cdot \Phi_-\cap \Phi_+|=|(w\cdot(\Phi_-\setminus \Phi_{-J})\cap \Phi_+)\sqcup (w\cdot \Phi_J\cap \Phi_+)|=|w\cdot(\Phi_-\setminus \Phi_{-J})\cap \Phi_+|+| w\cdot \Phi_J\cap \Phi_+|$.
Since $l(ww_{0,P})=l(w)+l(w_{0,P})$, we must have
$|w\cdot \Phi_{-J}\cap \Phi_+|+|\Phi_J|=|w\cdot \Phi_{J}\cap \Phi_+|$.
Therefore, $|w\cdot \Phi_{-J}\cap \Phi_+|=0$, $|w\cdot \Phi_{J}\cap \Phi_+|=|\Phi_J|$,
and the claim follows.
\end{proof}

\begin{lemma} Let $Q$ be a reduced word for $w$ where $w\in W^P$, then
\[
    \Im(m_Q) w^{-1} = w\cdot U^P_-\cap U.
\]
\end{lemma}

\begin{proof} 
Since $U_-^P=w_{0,P}\cdot U_- \cap U_-$, we have $w\cdot U^P_-\cap U = ww_{0,P}\cdot U_- \cap w\cdot U_- \cap U$. Let $J$ be the set of simple
roots that generate $P$, $\Phi_J$ be the
set of positive roots generated by
$J$, and $\Phi_{-J}$ be the set of negative roots generated by $-J$. Since $w\in W^P$, we have $w\cdot \Phi_{-J}\subset \Phi_-$ by Lemma 1.  
Therefore, $w\cdot \Phi_-\cap \Phi_+ =w\cdot(\Phi_{-J}\sqcup (\Phi_-\setminus \Phi_{-J}))\cap\Phi_+=w\cdot (\Phi_-\setminus \Phi_{-J})\cap\Phi_+$.
Moreover, $w_{0,P}\cdot \Phi_-=\Phi_J\sqcup (\Phi_-\setminus \Phi_{-J})$, so $ww_{0,P}\cdot\Phi_-\cap \Phi_+=w\cdot (\Phi_J\sqcup (\Phi_-\setminus \Phi_{-J}))\cap \Phi_+\supset w\cdot \Phi_-\cap \Phi_+$, which means
$w\cdot U_-\cap U\subset ww_{0,P}\cdot U_-\cap U$.
Therefore $w\cdot U_-^P\cap U = w\cdot U_-\cap U$,
which has dimension $l(w)$. The dimensions of both sides match up, so it suffices to show that for any
$\mathbf{z}$, $m_Q(\mathbf{z})w^{-1}\in w\cdot U_-\cap U$.

We proceed by induction on the length of $Q$. 
Let $\widecheck{\rho}$ be the sum of
fundamental coweights and we
use the fact that for any $x\in G$, $\lim_{t\to \infty}\widecheck{\rho}(t)\cdot x = 1$ iff $x \in U_-$. When $w=s_\alpha$, we need $\exp(zR_\alpha)\in s_\alpha\cdot U_-$ which is given by
$\lim_{t\to\infty} (s_\alpha\cdot \widecheck{\rho}(t))\cdot \exp(zR_\alpha)
=1$ since $s_\alpha\exp(zR_\alpha)s_\alpha\in U_{-\alpha}$.  
Clearly $\exp(zR_\alpha)\in U$, so the statement is
true when $w$ is a simple reflection.

Now suppose $s_\alpha w>w$. 
Notice that $s_\alpha$ preserves the set  $\Phi_-\setminus\{-\alpha\}$.
It follows that 
$s_\alpha\cdot(w\cdot\Phi_-\cap\Phi_-\setminus\{-\alpha\})\subset \Phi_-$.
In other words, all negative roots
in $w\cdot\Phi_-$ that are not $-\alpha$ remain negative under the
action of $s_\alpha$. Since $s_\alpha w>w$, 
$s_\alpha w\cdot \Phi_-$ contains more positive roots
than $w\cdot \Phi_-$ 
, it must be the case that
$\alpha\in (s_\alpha w)\cdot \Phi_-$, and thus
$U_\alpha \subset s_\alpha w\cdot U_-$. 
Assume by induction $m_Q(\mathbf{z})w^{-1}\in w\cdot U_-\cap U$,
we  show that $\exp(zR_\alpha)s_\alpha m_Q(\mathbf{z})w^{-1} s_\alpha \in (s_\alpha w)\cdot U_-\cap U$.  Now
\vskip 0.5em
\begin{tabular}{rl}
& $\lim_{t\to\infty}(s_\alpha w\cdot\widecheck{\rho}(t))\cdot (\exp(zR_\alpha)s_\alpha m_Q(\mathbf{z})w^{-1} s_\alpha)$\\
$=$&$\lim_{t\to\infty}(s_\alpha w\cdot\widecheck{\rho}(t))\cdot (\exp(zR_\alpha))\lim_{t\to\infty}(s_\alpha w\cdot\widecheck{\rho}(t))\cdot (s_\alpha m_Q(\mathbf{z})w^{-1} s_\alpha)$.
\end{tabular} 
\vskip 0.5em
\noindent The first factor
is 1 by $U_\alpha\subset s_\alpha w\cdot U_-$ and the
second factor is 1 by $\lim_{t\to\infty}(w\cdot\widecheck{\rho}(t))\cdot m_Q(\mathbf{z})w^{-1}=1$
(the induction hypothesis). This establishes 
$\exp(zR_\alpha)s_\alpha m_Q(\mathbf{z})w^{-1} s_\alpha \in (s_\alpha w)\cdot U_-$.

We are left to show $\exp(zR_\alpha)s_\alpha m_Q(\mathbf{z})w^{-1} s_\alpha \in U$. 
Since $m_Q(\mathbf{z})w^{-1}\in w\cdot U_-\cap U$,
$s_\alpha m_Q(\mathbf{z})w^{-1}s_\alpha\in  s_\alpha w\cdot U_-\cap s_\alpha\cdot U$. Since $\alpha\in s_\alpha w\cdot \Phi_-$ but $s_\alpha\cdot\Phi_+ = \Phi_+\setminus \{\alpha\}\cup\{-\alpha\}$,
$s_\alpha w\cdot \Phi_-\cap s_\alpha \cdot \Phi_+\subset \Phi_+\setminus\{\alpha\}$. It follows
that $s_\alpha m_Q(\mathbf{z})w^{-1}s_\alpha\in U$ and $\exp(zR_\alpha)s_\alpha m_Q(\mathbf{z})w^{-1} s_\alpha \in U$.
\end{proof}
Let $w_1,w_2\in W^P$ such that $w_2w_1=w_0w_{0,P}$
and $l(w_2)+l(w_1)=l(w_0w_{0,P})$. Choose
$Q_1$ a reduced word for $w_1$ and $Q_2$ a reduced word for $w_2$. Let $m_1, m_2$ denote the Bott-Samelson
maps $m_{Q_1}$, $m_{Q_2}$, respectively. 
\begin{lemma}
$w_2^{-1}\Im(m_{Q_2})=w_1\cdot U_-^P\cap U_-$.
\end{lemma}
\begin{proof}
The dimensions of both sides are $l(w_2)=l(w_0)-l(w_{0,P})-l(w_1)$, so it 
suffices to show that for any $\mathbf{z}$,
$w_2^{-1}m_{Q_2}(\mathbf{z})\in w_1\cdot U_-^P\cap U_-$, or $m_{Q_2}(\mathbf{z})^{-1}w_2\in w_1\cdot U_-^P\cap U_-$.
There exists a reduced word $Q_2'$ of 
$w_0w_2^{-1}w_0$ such that $w_0m_{Q_2'}(\mathbf{z})w_0=m_{Q_2}(\mathbf{z})^{-1}$. Our statement reduces to
$w_0m_{Q_2'}(\mathbf{z})w_0w_2\in w_1\cdot U_{-}^P\cap U_-$. Conjugating by $w_0$, we
get $m_{Q_2'}(\mathbf{z})w_0w_2w_0\in w_0w_1\cdot U^P_-\cap U$. Notice that
$w_0w_2^{-1}w_0=w_0w_1w_{0,P}\in W^P$, since
$l(w_0w_2^{-1}w_0)=l(w_2)=l(w_0)-l(w_{0,P})-l(w_1)$ and $l(w_0w_1)=l(w_0)-l(w_1)$. The
claim then follows from the previous lemma,
since $w_0w_1\cdot U_-^P=w_0w_1w_{0,P}\cdot U_-^P.$
\end{proof}

\begin{lemma}
$t^\lambda \cdot U^P\subset \mathcal{U}^{G(\mathcal{O}_-)}$, and
$t^\lambda \cdot U_-^P\subset \mathcal{U}^{G(\mathcal{O})}$.
\end{lemma}
\begin{proof}
The roots that correspond to $U^P$ are
$\Phi_+\setminus \Phi_J$. 
For any $\beta$ in this set,
$\langle\lambda,\beta\rangle >0$.
Then
$t^\lambda\cdot \beta = \beta+\langle\lambda,\beta\rangle \delta$, which is in the roots of
$\mathcal{U}^{G(\mathcal{O}_-)}$.
The proof for the second part is similar.
\end{proof}
We will now give  parametrizations of both sides
of the isomophisms separately. 
Let $(z_1,\cdots, z_{|Q_1|},z_{|Q_1|+1},z_{|Q_1|+|Q_2|})\in\mathbb{A}^{l(w_0w_{0,P})}$ be parameters. For ease of
notation, let $\mathbf{z}_1 =(z_1,\cdots,z_{|Q_1|})$, $\mathbf{z}_2 =(z_{|Q_1|+1,\cdots,|Q_1|+|Q_2|})$, and
$\mathbf{z} =(z_1,\cdots, z_{|Q_1|},z_{|Q_1|+1},z_{|Q_1|+|Q_2|})$.
To parametrize 
$\mathcal{X}_\circ^{w_1t^\lambda {w_1}^{-1}}\cap \mathcal{X}_{t^\lambda w_{0,P}w_0}=\mathcal{X}_\circ^{w_1t^\lambda w_{0,P}w_0w_2}\cap \mathcal{X}_{t^\lambda w_{0,P}w_0}$, use the map
\[\psi_{w_1}:\mathbf{z}\mapsto m_1(\mathbf{z}_1)t^\lambda w_{0,P}w_0m_2(\mathbf{z}_2)\mathcal{I}/\mathcal{I}.\]
We first see the image of this map  lies
inside $\mathcal{X}^{w_1t^\lambda w_{0,P}w_0w_2}_\circ$.
Since $t^\lambda w_{0,P}w_0$ is the smallest
element in the set $Wt^\lambda W$ \cite{HL}, $l(w_1t^\lambda w_{0,P}w_0w_2)=l(w_1)+l(t^\lambda w_{0,P}w_0)+l(w_2)$. Choose a reduced word
$Q'$ for $t^\lambda w_{0,P}w_0$, then
the Bott-Samelson map $m_{Q_1Q' Q_2}\circ\pi$
 where $\pi$ is the projection
$G(\mathcal{K})\to G(\mathcal{K})/\mathcal{I}$
parametrizes $\mathcal{X}^{w_1t^\lambda w_{0,P}w_0w_2}_\circ$. Setting the coordinates
that correspond to $Q'$ to 0 we get $\psi_{w_1}$.
The fact that that the image of 
$\psi_{w_1}$ lies in $\mathcal{X}_{t^\lambda w_{0,P}w_0}$ will follow from
the proof of Proposition 1 below.

Now we describe the parametrization of $w_1U_-^PP/P\subset G/P$. First we consider the parametrization of $w_1\cdot U_-^P$ 
\[p: \mathbb{A}^{l(w_{0,P}w_0)}\to w_1\cdot U_-^P,\]
\[ p(\mathbf{z})=m_1(\mathbf{z}_1)w_{0,P}w_0m_2(\mathbf{z}_2).\]
The fact that this is indeed
a parametrization follows from
Lemma 2 and 3.
Since $p(\mathbf{z})\in w_1U_-^Pw_1^{-1}$, there exists a unique
$u_-(\mathbf{z})\in U_-\cap w_1U_-^Pw_1^{-1}$ such that 
$u_-(\mathbf{z})p(\mathbf{z})\in U\cap w_1U_-^Pw_1^{-1}$
by Lemma 2.2 in \cite{KWY}.
The parametrization of $wU_-^PP/P\subset G/P$
is given by
\[\phi_{w_1}:\mathbf{z}\mapsto u_-(\mathbf{z})m_1(\mathbf{z}_1)P/P.\]
\begin{prop}
The map $\psi_{w_1}\circ \phi_{w_1}^{-1}$
is a stratified isomorphism on the $w_1$-charts.
\end{prop}

\begin{proof}
Assume $m_1(\mathbf{z}_1)=u_-(\mathbf{z})m_1(\mathbf{z}_1)P/P 
\in \mathring{\Pi}_v^w$ where $(w,v)\in W^J\times W$, $v\le w$.
This means there exists $p\in P$ such that $u_-(\mathbf{z})m_1(\mathbf{z}_1)p\in BwB\cap B_-vB$.
We may take $p\in L\cap U_-$.  
Let $u_+(\mathbf{z})=u_-(\mathbf{z})p(\mathbf{z})$, then  $u_-(\mathbf{z})m_1(\mathbf{z_1})=u_+(\mathbf{z})m_2(\mathbf{z}_2)^{-1}w_0w_{0,P}$.
Since $u_-(\mathbf{z})\in U_-$ and $u_+(\mathbf{z})\in U$, we have $m_1(\mathbf{z_1})p\in B_-vB$
and $p^{-1}w_{0,P}w_0m_2(\mathbf{z}_2)\in Bw^{-1}B$.

We  want to show that $m_1(\mathbf{z}_1)t^\lambda w_{0,P}w_0 m_2(\mathbf{z}_2)\mathcal{I}/\mathcal{I}\in\mathcal{X}_{vt^\lambda w^{-1}}$.
Since $p\in L\cap U_-$, $p$  commutes with $t^\lambda$, so 
$m_1(\mathbf{z}_1)t^\lambda w_{0,P}w_0 m_2(\mathbf{z}_2)=m_1(\mathbf{z}_1)pt^\lambda p^{-1} w_{0,P}w_0 m_2(\mathbf{z}_2)\in B_-vBt^{\lambda}Bw^{-1}B$.
We now argue that $ B_-vBt^{\lambda}Bw^{-1}B\subset \mathcal{I}_-vt^{\lambda}w^{-1}\mathcal{I}$.
It suffices to show that for all $b_1,b_2\in B$, $vb_1t^\lambda b_2 w^{-1}\in \mathcal{I}_-vt^\lambda w^{-1}\mathcal{I}$.

Write $b_1=l_1u_1$, where $l_1\in L\cap B$ and $u_1\in U^P$.  By Lemma 4, $u_1t^\lambda = t^\lambda u_1'$, where $u_1'\in \mathcal{U}^{G(\mathcal{O})}$. 
Since $b_2w^{-1}\in G$ normalizes $\mathcal{U}^{G(\mathcal{O})}$,  there exists
$u_1''\in \mathcal{U}^{G(\mathcal{O})}\subset\mathcal{I}$ such that $u_1'b_2w^{-1}= b_2w^{-1}u_1''$.

So far we have shown that $vb_1t^\lambda b_2 w^{-1} =vl_1t^\lambda b_2w^{-1}u_1''$. 
Since $l_1\in L\cap B$, $l_1$ commutes with $t^\lambda$. Therefore, $vl_1t^\lambda b_2w^{-1}u_1''= vt^\lambda b_2'w^{-1}u_1''$ for some $b_2'\in B$.
Now write $b_2'=u_2l_2$ where  $u_2\in U^P$ and  $l_2\in L\cap B$.   
Again by Lemma 4, we have $t^\lambda u_2 =u_2't^\lambda$ where $u_2'\in \mathcal{U}^{G(\mathcal{O}_-)}$. Since $v\in G$ normalizes
$\mathcal{U}^{G(\mathcal{O}_-)}$, we have $vu_2'=u_2''v$ for some $u_2''\in \mathcal{U}^{G(\mathcal{O}_-)}\subset\mathcal{I}_-.$ 
Furthermore, by Lemma 1, since $w\in W^J$, $w(L\cap B)w^{-1}\subset B$. 
Therefore, there exists some $b\in B$ such that $l_2 w^{-1}=w^{-1}b$. 
Summing up, we have $vt^\lambda u_2l_2 w^{-1}u_1''=u_2''vt^\lambda w^{-1}bu_1''\in \mathcal{I}_-vt^\lambda w^{-1}\mathcal{I}$ as desired.  
\end{proof}
\section{An Example in Type A}

    Let $G=SL_4(\C)$, $S=\{\alpha_1, \alpha_2, \alpha_3\}$, $J=\{\alpha_1,\alpha_3\}$, and $\lambda=(1,0,0,-1)$. Let $s_1, s_2,s_3$ be the simple
    reflections associated to the simple roots
    $\alpha_1,\alpha_2,\alpha_3$. In this case $w_0w_{0,P}=4231$ (in
    one-line notation).  The factors for
     Bott-Samelson maps are of the form
    \[\exp(zR_{\alpha_1})=\left[\begin{smallmatrix}z&1& & \\-1& & & \\ & & 1 & \\ & & & 1\end{smallmatrix}\right],\ 
    \exp(zR_{\alpha_1})=\left[\begin{smallmatrix}1& & & \\ &z&1& \\ & -1&  & \\ & & & 1\end{smallmatrix}\right],\ 
    \exp(zR_{\alpha_1})=\left[\begin{smallmatrix}1&& & \\& 1& & \\ & & z &1 \\ & &-1 & \end{smallmatrix}\right].\]
    We show an example of computations on the chart 
    $w_1U_-^PP/P$ where $w_1=4123=s_3s_2s_1$. 
    Then $w_2=w_0w_{0,P}w_1^{-1}=2314=s_1s_2$.
    We demonstrate that the chart isomorphism is stratification
    preserving by showing explicitly that the equations
    of the projected Richardson divisors that meet
    $w_1U_-^PP/P$  match the corresponding equations of the
    Schubert divisors on the Kazhdan-Lusztig variety $\mathcal{X}^{w_1t^{\lambda}w_1^{-1}}_\circ \cap \mathcal{X}_{t^\lambda w_0w_{0,P}}\subset \widehat{SL_4}(\C)/\mathcal{I}$. This is enough, because the stratifications 
    on both sides are
    uniquely determined by divisors.
    
    The chart $w_1U_-^P/P$ is parametrized as follows
    \[\phi_{w_1}:(z_1,z_2,z_3,z_4,z_5)\mapsto
 \begin{bmatrix}z_3&1& & \\z_2& & 1 & \\ z_1+z_3z_4+z_2z_5&z_4 & z_5 &1 \\ -1 & & & \end{bmatrix}P/P.\]
    The codimension one projected Richardson divisors
    on this $G/P$ are $\Pi_{2134}^{4231}$, $\Pi_{1324}^{4231}$, $\Pi_{1243}^{4231}$, $\Pi_{1234}^{4132}$,  and $\Pi_{1234}^{3241}$.
    Among these, the first four have nonempty intersections with
    our chart under consideration. Their equations on this chart
    are $z_3$, $z_3z_4+z_2z_5$, $z_1$,  $z_4$, respectively.
    
    Meanwhile, the Kazhdan-Lusztig variety 
    $\mathcal{X}^{w_1t^{\lambda}w_1^{-1}}_\circ \cap \mathcal{X}_{t^\lambda w_0w_{0,P}}$ is parametrized as follows:
    
    \begin{tabular}{rl}
    $\psi_{w_1}:(z_1,z_2,z_3,z_4,z_5)\mapsto$ & $ 
    \begin{bmatrix}
    z_3&1& & \\
    z_2& &1 & \\
    z_1& &  & 1\\
    -1& & & \end{bmatrix}
    \begin{bmatrix}& & &t \\ & -1& & \\ & & -1 & \\1/t & & & \end{bmatrix}
    \begin{bmatrix}z_4&z_5& 1& \\-1& & & \\ &-1 &  & \\ & & & 1\end{bmatrix}\mathcal{I}/\mathcal{I}$\\
    & \\
    $=$& $
    \begin{bmatrix}
    1& & &tz_3 \\ &1 & &tz_2 \\ -z_4t^{-1}&-z_5t^{-1} &-t^{-1} &tz_1 \\ & & & -t\end{bmatrix}\mathcal{I}/\mathcal{I}$\\
    \end{tabular}
    
    Following the rules given in \cite{H}, we compute
    that the equations 
    for $\mathcal{X}^{w_1t^{\lambda}w_1^{-1}}_\circ \cap \mathcal{X}_{s_1t^\lambda w_0w_{0,P}}$, $\mathcal{X}^{w_1t^{\lambda}w_1^{-1}}_\circ \cap \mathcal{X}_{s_2t^\lambda w_0w_{0,P}}$,
    $\mathcal{X}^{w_1t^{\lambda}w_1^{-1}}_\circ \cap \mathcal{X}_{s_3t^\lambda w_0w_{0,P}}$, and
    $\mathcal{X}^{w_1t^{\lambda}w_1^{-1}}_\circ \cap \mathcal{X}_{t^\lambda w_0w_{0,P}s_1}$ are also
    $z_3$, $z_3z_4+z_2z_5$, $z_1$,  $z_4$, respectively.
    
    \bibliographystyle{alpha}
  \bibliography{ref}
    
    \end{document}